\documentclass[11pt]{article}  

\usepackage{amsmath, amsthm, amsfonts, amssymb}
\usepackage[bookmarks]{hyperref}
\usepackage{indentfirst} 
\usepackage{marvosym}

\usepackage[a4paper]{geometry}

\newtheorem{lemma}{Lemma}[section]
\newtheorem{theorem}[lemma]{Theorem}
\newtheorem{proposition}[lemma]{Proposition}
\newtheorem{corollary}[lemma]{Corollary}
\renewenvironment{proof}[1][\proofname]{{\sc #1. }}{\qed}
\newtheorem{theoremletters}{Theorem}

\newtheorem{corollaryletters}[theoremletters]{Corollary}

\theoremstyle{definition}
\newtheorem{remark}[lemma]{Remark}
\newtheorem{example}[lemma]{Example}
\newtheorem{defin}[lemma]{Definition}{\bf}{\rm}

\newcommand{\abs}[1]{\ensuremath{\left| #1 \right|}}
\newcommand{\op}{\operatorname}
\newcommand{\ce}[2]{\pmb{\op{C}}_{#1}(#2)}
\newcommand{\no}[2]{\pmb{\op{N}}_{#1}(#2)}
\newcommand{\ze}[1]{\pmb{\op{Z}}(#1)}
\newcommand{\fra}[1]{\pmb{\op{\Phi}}(#1)}
\newcommand{\fit}[1]{\pmb{\op{F}}(#1)}
\newcommand{\rad}[2]{\pmb{\op{O}}_{#1}(#2)}
\newcommand{\syl}[2]{\op{Syl}_{#1}\left(#2\right)}
\newcommand{\hall}[2]{\op{Hall}_{#1}\left(#2\right)}

\begin{document}

\title{\bf Products of groups and class sizes of $\pi$-elements}

\author{\sc M. J. Felipe $\cdot$ A. Mart\'inez-Pastor $\cdot$ V. M. Ortiz-Sotomayor
\thanks{The first author is supported by Proyecto Prometeo II/2015/011, Generalitat Valenciana (Spain). The second author is supported by Proyecto MTM 2014-54707-C3-1-P, Ministerio de Econom\'ia, Industria y Competitividad (Spain), and by Proyecto Prometeo/2017/057, Generalitat Valenciana (Spain). The results in this paper are part of the third author's Ph.D. thesis, and he acknowledges the predoctoral grant ACIF/2016/170, Generalitat Valenciana (Spain).\newline\rule{6cm}{0.1mm}\newline
Instituto Universitario de Matemática Pura y Aplicada (IUMPA-UPV), Universitat Polit\`ecnica de Val\`encia, Camino de Vera s/n, 46022 Valencia, Spain\newline
\Letter: \texttt{mfelipe@mat.upv.es}, \texttt{anamarti@mat.upv.es}, \texttt{vicorso@doctor.upv.es} \newline
ORCID iDs: 0000-0002-6699-3135, 0000-0002-0208-4098, 0000-0001-8649-5742
}}

\date{}

\maketitle

\begin{abstract}
\noindent We provide structural criteria for some finite factorised groups $G=AB$ when the conjugacy class sizes in $G$ of certain $\pi$-elements in $A\cup B$ are either $\pi$-numbers or $\pi'$-numbers, for a set of primes $\pi$. In particular, we extend for products of groups some earlier results.

\medskip


\smallskip

\end{abstract}


\section{Introduction}

Along the last decades, numerous researchers have investigated groups which can be factorised as the product of two subgroups. In this setting, one of the main goals is to study the influence that the structure of the factors has on the structure of the whole group (and vice versa). In some occasions, the imposition of certain permutability conditions on the subgroups in the factorisation has been revealed very useful in that task. A detailed account on this topic can be found in the book \cite{BEA}. Throughout this paper, we deal with products of groups that possess a especial chief series, the so-called \emph{core-factorisations}, introduced in \cite{FMOvan} (see also Definition \ref{definition_core}).

On the other hand, a current activity shows up that, in a factorised group, the sizes of the conjugacy classes in the group of the elements in the factors have a strong impact on the structure of the whole group (see, for instance, \cite{BCL, FMOprime, FMOvan, ZGS}). Our main purpose here is to study the $\pi$-structure of groups with a core-factorisation when the class lengths in the group of the $\pi$-elements in the factors are either $\pi$-numbers or $\pi'$-numbers, for a set of primes $\pi$. In fact, we present alternative proofs of some earlier results as consequences of our theorems when trivial factorisations are considered. We point out that Dolfi (\cite{D}) analysed this last property on the class lengths in a (not necessarily factorised) group, but focusing on all its elements (not only on those with order a $\pi$-number). It is worth also to highlight that, although some results on class sizes could be proved through elementary arguments, the analysis of the considered class size properties in the context of products of groups may need a wider approach, even for core-factorisations. Indeed, in a core-factorisation, there is no relation in general between the class size of an element in one factor and the size of the corresponding conjugacy class in the whole group, in contrast to other developments (see, for example, \cite{ZGS}).

The paper is structured in the following way: in Section \ref{sec_core} we gather the definition and some properties of core-factorisations. Later on, in Section \ref{ppo-pi-elements} we analyse groups with a core-factorisation such that the class lengths in the whole group of $\pi$-elements of prime power order in the factors are $\pi$-numbers (Theorem \ref{pi-decompo}). Then, in Section \ref{pi-elements}, we put focus on groups with a core-factorisation whose $\pi$-elements (not necessarily of prime power order) in the factors have class sizes equal to either $\pi$-numbers or $\pi'$-numbers (Theorem \ref{pi-pi'}, and Corollary \ref{corPiPi'} for not necessarily factorised groups). We also analyse in Theorem \ref{teosilvio} this last condition on the class sizes of every element in the factors. Finally, as a consequence of the previous results, prime power class sizes are studied for the $\pi$-elements in the factors of a group with a core-factorisation (Theorem \ref{teoprime}). We want to remark that, along the whole paper, we provide numerous examples which show the scope of the results presented.

In the sequel, all groups under consideration are finite. For a group $G$ and an element $x\in G$, we denote by $x^G$ the conjugacy class of $x$ in $G$, and its size is $\abs{x^G}=\abs{G:\ce{G}{x}}$. We represent the set of all prime divisors of a natural number $n$ by $\pi(n)$, and in particular we use $\pi(G)$ for the set of all prime divisors of the order of $G$. The set of all Hall $\pi$-subgroups of $G$ is expressed by $\hall{\pi}{G}$, where $\pi$ will always denote a set of primes. A group such that $G=\rad{\pi}{G} \times \rad{\pi'}{G}$ is said to be $\pi$-decomposable. By $\op{core}_X(H)$ we mean the core in a group $X$ of a subgroup $H$, i.e. the largest normal subgroup of $X$ contained in $H$. Given a group $G=AB$ which is the product of the subgroups $A$ and $B$, a subgroup $S$ is called prefactorised (with respect to this factorisation) if $S=(S\cap A)(S\cap B)$ (see \cite{AMB}). We recall that a subgroup $U$ covers a section $V/W$ of a group $G$ if $W(U\cap V)= V$. As usual, CFSG will denote the classification of finite simple groups. The remaining notation and terminology is standard within the theory of groups, and it is taken mainly from \cite{DH}. We also refer to this book for details about classes of groups.


\section{On core-factorisations}
\label{sec_core}

As mentioned in the introduction, along the paper we deal with a especial kind of products of groups, the \emph{core-factorisations}. We start this section by introducing that concept. Besides our initial inspiration in \cite{FMOvan} within the framework of products of groups with certain permutability conditions on the factors, this notion can be also motivated by the following observation: If $\pi$ is a set of primes and $G$ is a group that possesses both Hall $\pi$-subgroups and Hall $\pi'$-subgroups, say $H$ and $L$ respectively, then $G=HL$ is a $\pi$-separable group if and only if for a chief series of $G$ it holds that all the chief factors are covered by either $H$ or $L$.

\begin{defin}
Let $1\neq G=AB$ be the product of the subgroups $A$ and $B$. We say that $G=AB$ is a \textbf{core-factorisation} whenever $G$ possesses a chief series such that each chief factor of $G$ is covered by either $A$ or $B$.
\label{definition_core}
\end{defin}

We point out that this definition of a core-factorisation is equivalent to that given in \cite{FMOvan} (see Lemma \ref{core_charac} below). Next we collect some of its properties, some of which appear in the cited paper.

\begin{remark}
Let us state some remarkable facts:
\begin{enumerate}
	\item[(i)] If either $1\neq G=A$ or $1\neq G=B$, then $G=AB$ is always a core-factorisation. 
	
	\item[(ii)] If $G=AB$ is a core-factorisation, then there exists always a minimal normal subgroup of $G$ contained in either $A$ or $B$.
	
	\item[(iii)] \cite[Example 1]{FMOvan} Every (totally) mutually permutable product of two subgroups is a core-factorisation.
	
	\item[(iv)] By the initial paragraph, if $G$ is $\pi$-separable, then $G=HL$ is a core-factorisation for any $H\in\hall{\pi}{G}$ and $L\in\hall{\pi'}{G}$.
\end{enumerate}
\end{remark}

Notice that, in the last above statement, the property of having coprime orders for the subgroups in the factorisation is essential, as the next example shows.

\begin{example}
Let $G$ be a symmetric group of $4$ letters. Then $G=AB$ where $A=\langle (1,3,2,4), (1,2)(3,4)\rangle$ and $B=\langle (3,4), (2,3,4)\rangle$. Note that $G$ is clearly $\pi(A)$-separable (indeed it is soluble), but the unique minimal normal subgroup of $G$ is not covered by either $A$ or $B$, so $G=AB$ is not a core-factorisation.
\end{example}

We also provided a useful characterisation of core-factorisations via quotients (compare with \cite[Lemma 2]{FMOvan}).

\begin{lemma}
\label{core_charac}
Let $1\neq G=AB$ be the product of the subgroups $A$ and $B$. The following statements are pairwise equivalent:
\begin{enumerate}
	\item[\emph{(1)}] $G=AB$ is a core-factorisation.
	
	\item[\emph{(2)}] There exists a normal series $1=N_0 \unlhd N_1 \unlhd \cdots \unlhd N_{n-1} \unlhd N_n=G$ such that either $N_i/N_{i-1} \leqslant AN_{i-1}/N_{i-1}$ or $N_i/N_{i-1} \leqslant BN_{i-1}/N_{i-1}$, for each $1\leq i \leq n$ (i.e. $N_i/N_{i-1}$ is covered by either $A$ or $B$).
	
	\item[\emph{(3)}] For every proper normal subgroup $K$ of $G$ it holds that there exists a normal subgroup $1\neq M/K$ of $G/K$ such that either $M/K\leqslant AK/K$ or $M/K\leqslant BK/K$ (i.e. either $A$ or $B$ covers $M/K$).
\end{enumerate}  
Further, in \emph{(1)} and \emph{(2)}, each term $N_i$ of such (chief) normal series is prefactorised and $N_i=(N_i\cap A)(N_i\cap B)$ is also a core-factorisation.
\end{lemma}

If we adopt the bar convention in statement (3) for the quotients over $K$, we point out that this condition means $\op{core}_{\overline{G}}(\overline{A}) \op{core}_{\overline{G}}(\overline{B})\neq 1$. This illustrates the given name for such factorisations.

Moreover, as noted in \cite[Example 2]{FMOvan}, if $N$ is an arbitrary prefactorised normal subgroup of a core-factorisation $G=AB$, then $N=(N\cap A)(N\cap B)$ may not be a core-factorisation. Nevertheless, this condition behaves well for quotients of $G$.

\begin{lemma}\emph{\cite[Lemma 1]{FMOvan}}
\label{lemacore}
Let $G=AB$ be a core-factorisation, and let $M$ be a proper normal subgroup of $G$. Then $G/M=(AM/M)(BM/M)$ is also a core-factorisation.
\end{lemma}

Next we show some series constructions for core-factorisations somehow similar to the lower/upper $\pi$-series of a $\pi$-separable group. Let $G=AB$ be the product of two subgroups $A$ and $B$. We can consider $\mathfrak{C}_A(G):=\op{core}_G(A)$. Next we take $G/\mathfrak{C}_A(G)=(A/\mathfrak{C}_A(G))(B\mathfrak{C}_A(G)/\mathfrak{C}_A(G))$. Since $\op{core}_{G/\mathfrak{C}_A(G)}(A/\mathfrak{C}_A(G))=1$, then we compute \break $\op{core}_{G/\mathfrak{C}_A(G)}(B\mathfrak{C}_A(G)/\mathfrak{C}_A(G))$. Let us denote its inverse image in $G$ by $\mathfrak{C}_{A,B}(G)$, i.e. $\mathfrak{C}_{A,B}(G)/\mathfrak{C}_A(G):=\op{core}_{G/\mathfrak{C}_A(G)}(B\mathfrak{C}_A(G)/\mathfrak{C}_A(G))$. Similarly we define $\mathfrak{C}_{A,B,A}(G)$ to be the inverse image in $G$ of $\op{core}_{G/\mathfrak{C}_{A,B}(G)}(A\mathfrak{C}_{A,B}(G)/\mathfrak{C}_{A,B}(G))$. Continuing these definitions in the natural way, we can obtain a sequence of normal subgroups of $G$ $$1\unlhd \mathfrak{C}_A(G) \unlhd \mathfrak{C}_{A,B}(G) \unlhd \mathfrak{C}_{A,B,A}(G) \unlhd \cdots.$$ We call this series the \textbf{core $A$-series} of $G$. Similarly, we define the \textbf{core $B$-series} of $G$ to be $$1\unlhd \mathfrak{C}_B(G) \unlhd \mathfrak{C}_{B,A}(G) \unlhd \mathfrak{C}_{B,A,B}(G) \unlhd \cdots.$$ In a core-factorisation $G=AB$, both series terminate in $G$ in virtue of Lemmas \ref{core_charac} and \ref{lemacore}; and conversely, if one of the two series terminates in $G$, then $G=AB$ is certainly a core-factorisation. Analogously, it is possible to define a \emph{core-length} with respect to either $A$ or $B$ in the same way as the $\pi$ or $\pi'$-length of a $\pi$-separable group.

A well-known result asserts that $\rad{\pi}{G/\rad{\pi'}{G}}$ is self-centralising in $G/\rad{\pi'}{G}$, for any $\pi$-separable group $G$. The next example examines the analogous phenomenon in core-factorisations.

\begin{example}
Let $G=\op{Sym}(4) \times \langle x\rangle$, where $\op{Sym}(4)$ denotes the symmetric group of $4$ letters and $o(x)=2$. If $A=\langle ((1,2), x),\: ((3,4), x), \: ((1,3)(2,4), x)\rangle$ and $B=\langle ((2,3,4), 1), \: ((3,4), 1), \:(1, x)\rangle$, then $G=AB$ is a core-factorisation. Moreover, $\mathfrak{C}_A(G)=1$ but $\mathfrak{C}_B(G)$ is not self-centralising.
\end{example}

On the other hand, the next result on Hall $\pi$-subgroups of $\pi$-separable factorised groups is a key step for our development (indeed the $\pi$-separability condition can be relaxed to the $D_{\pi}$-property, as can be seen in \cite[1.3.2]{AMB}).

\begin{lemma}
\label{prefact}
Let the $\pi$-separable group $G = AB$ be the product of the subgroups $A$ and $B$. Then there exists a Hall $\pi$-subgroup $H$ of $G$ such that $H = (H\cap A)(H\cap B)$, with $H\cap A$ a Hall $\pi$-subgroup of $A$ and $H\cap B$ a Hall $\pi$-subgroup of $B$.

In particular, if $G=AB$ is a core-factorisation, then $H = (H\cap A)(H\cap B)$ is also a core-factorisation.
\end{lemma}

\begin{proof}
The first assertion is just a reformulation of \cite[1.3.2]{AMB}, so we concentrate on the second claim. We assume that $G=AB$ is a core-factorisation, and let us prove that $H = (H\cap A)(H\cap B)$ so is. There exists a chief series $1=N_0 \unlhd N_1 \unlhd \cdots \unlhd N_{n-1} \unlhd N_n=G$ such that each $N_i/N_{i-1}$ is covered by either $A$ or $B$, for all $1\leq i\leq n$. 

Note that $1=N_0 \cap H\unlhd N_1\cap H \unlhd \cdots \unlhd N_{n-1} \cap H \unlhd N_n\cap H=H$ is a normal series of $H$. We claim that each $(N_i \cap H)/(N_{i-1} \cap H)$ is covered by either $H\cap A$ or $H\cap B$, in order to apply Lemma \ref{core_charac}. Since $G$ is $\pi$-separable, then $N_i/N_{i-1}$ is either a $\pi$-group or a $\pi'$-group. In the latter case, we easily get that $(N_i \cap H)/(N_{i-1} \cap H)=1$ is clearly covered by either $H\cap A$ or $H\cap B$. If $N_i/N_{i-1}$ is a $\pi$-group, since we may assume for instance that $N_i/N_{i-1}$ is covered by $A$, then $N_i/N_{i-1}\leqslant (H\cap A)N_{i-1}/N_{i-1}$ because $H\cap A\in\hall{\pi}{A}$. Now $N_i=N_{i-1}(N_i \cap H\cap A)$ and $H\cap N_i  = H\cap N_{i-1}(N_i \cap H\cap A)=(N_i \cap H\cap A)(H\cap N_{i-1})\leqslant (H\cap A)(H\cap N_{i-1})$. Thus $(H\cap N_i)/(H\cap N_{i-1})$ is covered by $H\cap A$ and we are done.
\end{proof}


\section{On conjugacy class sizes of prime power order \texorpdfstring{$\pi$}{pi}-elements}
\label{ppo-pi-elements}

Along this section, we will consider products of groups and prime power order $\pi$-elements in the factors whose class sizes are $\pi$-numbers. First of all, we start by presenting some preliminary results. The next elementary properties are used frequently and without further reference.

\begin{lemma}
Let $N$ be a normal subgroup of a group $G$, and $A$ be a subgroup of $G$. We have:
\begin{itemize}
	\item[\emph{(a)}] $\abs{x^N}$ divides $\abs{x^G}$, for any $x\in N$.
	
	\item[\emph{(b)}] $\abs{(xN)^{G/N}}$ divides $\abs{x^G}$, for any $x\in G$.
	
	\item[\emph{(c)}] If $xN$ is a $\pi$-element of $AN/N$, then there exists a $\pi$-element $x_1\in A$ such that $xN = x_1N$.
\end{itemize}
\end{lemma}

The next observation is crucial in the sequel.

\begin{remark}
\label{values}
Let $G=AB$ be a $\pi$-separable group. Consider by Lemma \ref{prefact} a Hall $\pi$-subgroup $H = (H\cap A)(H\cap B)$ of $G$ such that $H\cap A\in\hall{\pi}{A}$ and $H\cap B\in\hall{\pi}{B}$. Then, imposing arithmetical conditions on the class sizes of the (prime power order) $\pi$-elements in $A\cup B$ is equivalent to impose them on the class sizes of the (prime power order) elements in $(H\cap A)\cup(H\cap B)$, because of the conjugacy of Hall $\pi$-subgroups.
\end{remark}

The lemma below is a transcription of a well-known Wielandt's result for a set of primes $\pi$.

\begin{lemma}\emph{\cite[Lemma 1]{BK}}
\label{wielandt}
Let $G$ be a group and $H\in\hall{\pi}{G}$. If $\abs{x^G}$ is a $\pi$-number for some $x\in H$, then $x\in\op{O}_{\pi}(G)$.
\end{lemma}

Indeed, we can provide a useful $\pi$-separability criterion for factorised groups having a Hall $\pi$-subgroup by means of the previous class size condition.

\begin{lemma}
\label{lemma_separability}
Assume that $G=AB$ with $\hall{\pi}{G}\neq \emptyset$. If $\abs{x^G}$ is a $\pi$-number for every $\pi$-element $x\in A\cup B$ of prime power order, then $\rad{\pi}{G}\in\hall{\pi}{G}$. In particular, $G$ is $\pi$-separable.
\end{lemma}

\begin{proof}
Let $H\in\hall{\pi}{G}$. We choose $P=(P\cap A)(P\cap B)\in\syl{p}{G}$ for some $p\in\pi$. Clearly $P$ is $G$-conjugate to a Sylow $p$-subgroup of $H$, say $P_1$. Hence for some $g\in G$ we have $P= P_1^g\leqslant H^g\in\hall{\pi}{G}$. It follows by Lemma \ref{wielandt} that if $x\in (P\cap A)\cup (P\cap B)$, then $x\in\rad{\pi}{G}$, so $(P\cap A)\cup (P\cap B)\subseteq \rad{\pi}{G}$. Thus $P^h\leqslant\rad{\pi}{G}$ for all $h\in G$. Since this is valid for all $p\in\pi$, we deduce that $H=\rad{\pi}{G}$. The second claim follows directly.
\end{proof}

\medskip

Further, under the additional assumption of being a core-factorisation, we get as our first main result a characterisation of the $\pi$-decomposability of such a factorised group. Our proof involves the following lemma, which makes use of the knowledge on the automorphism groups of the non-abelian simple groups.

\begin{lemma}\emph{\cite[Lemma 2.6]{DPSS}}
Let $S$ be a simple group. Then there exists $r\in \pi(S)$ such that $\op{gcd}(r, \abs{\ce{S}{\alpha}})=1$ for every non-trivial $\alpha \in \op{Aut}(S)$ of order coprime to $S$.
\label{aut_lemma}
\end{lemma}

\begin{theoremletters}
\label{pi-decompo}
Let $G=AB$ be a core-factorisation such that $\hall{\pi}{G}\neq \emptyset$. Then:
\begin{enumerate}
	\item[\emph{(1)}] Each $\abs{x^G}$ is a $\pi$-number for every $\pi$-element $x\in A\cup B$ of prime power order if and only if $G$ is $\pi$-decomposable.
	
	\item[\emph{(2)}] Each $\abs{x^G}$ is a $\pi$-number for every prime power order element $x\in A\cup B$ if and only if $G$ is $\pi$-decomposable and its Hall $\pi'$-subgroup is abelian.
\end{enumerate}
\end{theoremletters}

\begin{proof}
(1) The sufficient condition is clear. Let us prove that $G=AB$ is $\pi$-decomposable whenever every $\abs{x^G}$ is a $\pi$-number for each $\pi$-element $x\in A\cup B$ of prime power order.  Take $G$ a minimal counterexample to the assertion. In virtue of Lemma \ref{lemma_separability} we can affirm that $H:=\rad{\pi}{G}\in\hall{\pi}{G}$, so $G$ is $\pi$-separable. Applying Lemma \ref{prefact}, we can choose $F\in \hall{\pi'}{G}$ prefactorised. Take $y\in F\cap A$. We claim that $G_y:=H\langle y \rangle$ satisfies the hypotheses of the theorem. We have $$ G_y=\langle y \rangle (H\cap A)(H\cap B)\subseteq (G_y\cap A)(G_y\cap B)\subseteq G_y,$$ so $G_y$ is prefactorised and $\hall{\pi}{G_y}=\{H\}\neq \emptyset$. Now we take $p\in \pi$ and $P$ a prefactorised Sylow $p$-subgroup of $G_y$ as in Lemma \ref{prefact}. Any element $x\in P\cap G_y\cap A=P\cap A$ has $\abs{x^G}$ a $\pi$-number. Hence, there exists $g\in G$ such that $F^g\leqslant \ce{G}{x}$. We can assume $g\in H$ because $G=HF$. Since $\langle y \rangle \leqslant F$, we get $\langle y \rangle^g \leqslant \ce{G_y}{x}$, so $\abs{x^{G_y}}$ is a $\pi$-number. This is analogously valid for the elements in $P\cap G_y\cap B=P\cap B$, and for all $p\in \pi$. Now Remark \ref{values} provide that all $\pi$-elements in $(G_y\cap A)\cup(G_y\cap B)$ of prime power order have conjugacy class size in $G_y$ a $\pi$-number. It remains to show that $G_y=(G_y\cap A)(G_y\cap B)$ is a core-factorisation. If we reproduce the techniques in the proof of Lemma \ref{prefact}, we get that $1=N_0 \cap H\unlhd N_1\cap H \unlhd \cdots \unlhd N_{n-1} \cap H \unlhd N_n\cap H=H$ is a normal series of $H$ such that each $(N_i \cap H)/(N_{i-1} \cap H)$ is covered by either $H\cap A\leqslant G_y\cap A$ or $H\cap B\leqslant G_y\cap B$. But $H$ and all the $N_i$ are normal in $G$, so $1=N_0 \cap H\unlhd N_1\cap H \unlhd \cdots \unlhd N_{n-1} \cap H \unlhd N_n\cap H=H\unlhd H\langle y\rangle=G_y$ is a normal series of $G_y$. Moreover, $G_y/H$ is clearly covered by $G_y\cap A$ because $y\in G_y\cap A$. Thus all the factors are covered by either $G_y\cap A$ or $G_y \cap B$ and  $G_y=(G_y\cap A)(G_y\cap B)$ is a core-factorisation by Lemma \ref{core_charac}. If $G_y<G$, then it follows by minimality that $H\leqslant \ce{G}{y}$. Hence, we can suppose for some $y\in(F\cap A)\cup(F\cap B)$ that $G=H\langle y \rangle$; otherwise $H\leqslant \ce{G}{F}$, a contradiction. Indeed, by the decomposition of $y$ as product of prime power order elements, the same arguments apply and we can assume that $\op{o}(y)$ is a $q$-number for some prime number $q\in \pi'$. 

Since the hypotheses are inherited by quotients of $G$, and the class of $\pi$-decomposable groups is a saturated formation, we may assume $\fra{G}=1$ and that there exists a unique minimal normal subgroup $N$ of $G$, so $N\leqslant H$. Thus $\rad{\pi'}{G}=1$. As $G/N$ is $\pi$-decomposable, then $\langle y \rangle N\unlhd G$, and $[H, \langle y\rangle] =[H, y]\leqslant N$. Moreover, by coprime action we get $H = [H, y]\ce{H}{y}\leqslant N\ce{H}{y}$, so $G=H\langle y \rangle=N\ce{G}{y}$. 

We claim that $G=N\langle y \rangle$. Set $T:=N\langle y \rangle$ and, contrariwise, we assume that $T<G$. Note that $\hall{\pi}{T}=\{N\}\neq\emptyset$. Since $G=AB$ is a core-factorisation and $N$ is the unique minimal normal subgroup of $G$, we may suppose that $N\leqslant A$. As $y\in (F\cap A) \cup (F\cap B)$, then clearly $T=(T\cap A)(T\cap B)$. If we consider the normal series $1\unlhd N \unlhd N\langle y\rangle=T$, then the factors are covered by either $T\cap A$ or $T\cap B$ and $T=(T\cap A)(T\cap B)$ is a core-factorisation by Lemma \ref{core_charac}. Moreover, the class size conditions are inherited by $T$ because it is prefactorised and normal in $G$. By minimality we obtain that $N\leqslant \ce{G}{y}$ and $G=N\ce{G}{y}=\ce{G}{y}$, which leads to a contradiction. Therefore, $G=N\langle y\rangle$.

Next we demonstrate that $N$ is non-abelian. Otherwise $N=\ce{G}{N}$ because of standard group theoretic arguments (\cite[Theorem A - 10.6]{DH}). By coprime action we get $N=[N, y]\times \ce{N}{y}$. But $\ce{N}{y}$ is normal in $G$ and $N=\ce{G}{N}$, so necessarily $\ce{N}{y}=1$. Since $N$ is $t$-elementary abelian for some prime $t\in \pi$, any non-trivial element $n\in N\leqslant A$ satisfies that $\abs{n^G}$ is a $\pi$-number, so a $G$-conjugate of $n$ lies in $\ce{N}{y}=1$, a contradiction.

Thus $N$ is non-abelian and we can write $N=L_1 \times \cdots \times L_k$ where all $L_i$ are isomorphic non-abelian simple groups and they form a full $G$-conjugacy class of subgroups. Since $G=N\langle y \rangle$, then $\langle y \rangle $ acts transitively on $\{ L_1, \ldots , L_k\}$. As $\langle y\rangle$ is a $q$-group with $q\in\pi'$, if $k>1$, then we get a contradiction because $\abs{N}=k\abs{L_1}$ and $N$ is a $\pi$-group. It follows that $k=1$ and $N$ is simple. Now we can apply Lemma \ref{aut_lemma} in order to affirm that there exists a prime $s\in \pi(N)$ such that $s$ does not divide $\abs{\ce{N}{y}}$. Let $x$ be a non-trivial $s$-element in $N\leqslant A$. Since by hypotheses there is a conjugate of $x$ which lies in $\ce{N}{y}=1$, we have reached the final contradiction. The proof of (1) is now completed.

(2) It is enough to show the necessity condition. Assume that $\abs{x^G}$ is a $\pi$-number for every prime power order element $x\in A\cup B$. Clearly, $G$ is $\pi$-decomposable by (1). Moreover, its unique Hall $\pi'$-subgroup $\rad{\pi'}{G}$ is prefactorised by Lemma \ref{prefact}. Since $\rad{\pi'}{G}\cap A$ and $\rad{\pi'}{G}\cap B$ are generated by prime power order elements, all of which lying in $\ze{\rad{\pi'}{G}}$ due to the class size assumptions, then $\rad{\pi'}{G}$ is abelian.
\end{proof}

\medskip

A question which remains open is whether the hypothesis of being a core-factorisation in Theorem \ref{pi-decompo} can be eliminated. Moreover, when we consider the trivial factorisation $G=A=B$ in the above theorem, we retrieve the next result in \cite{ZGS}. In fact, our arguments provide an alternative proof. We remark that the proof given in that paper uses deeply the CFSG via a result due to Fein, Kantor and Schacher (see \cite[Lemma 2]{ZGS}).

\begin{corollary}\emph{\cite[Theorem 3.1]{ZGS}}
Let $G$ be a group with $\hall{\pi}{G}\neq \emptyset$. Then each $\abs{x^G}$ is a $\pi$-number for every $\pi$-element $x\in G$ of prime power order if and only if $G$ is $\pi$-decomposable.
\label{pi-decompo_ZGS}
\end{corollary}

\begin{remark}
Actually, when all the $\pi$-elements are considered in the above result (not only those of prime power order), then the CFSG can be avoided (see either \cite[Supplement to Theorem 1]{BK} or Lemma \ref{BKlemma} below).
\end{remark}

Zhao \emph{et al.} also provided in \cite[Theorem 3.2]{ZGS} a similar characterisation to the one in Theorem \ref{pi-decompo}, but considering a factorised group $G=AB$ with one factor which is subnormal. It is worth to remark that, if $A$ is subnormal, then for every element $x\in A$ it holds that $\abs{x^A}$ divides $\abs{x^G}$, although in general this is not the case. Besides, there exists a normal subgroup of $G$ which contains $A$, so this normal subgroup is prefactorised.

\begin{example}
Notice that, a priori, groups with a core-factorisation and factorised groups with one subnormal factor are not related. For instance, let $G$ be the natural wreath product of a symmetric group of degree $3$ and a cyclic group $\langle z \rangle$ of order $2$. If we take $A =\langle (2,3), (1,2,3)^z, (2,3)^z\rangle$ and $B = \langle (1,3,2)(4,5,6)^z, (1,3,2)(4,5,6)^zz\rangle$, then $G = AB$ is not a core-factorisation and $B$ is subnormal in $G$. On the other hand, it is not difficult to find core-factorisations where the factors are neither subnormal in the whole group nor mutually permutable (\cite[Example 2]{FMOvan}).
\end{example}

Next, we deal with the dual condition on the class sizes of prime power order $\pi$-elements, i.e. when they are not divisible by any prime in $\pi$. We characterise arbitrary factorisations of $\pi$-separable groups which have abelian Hall $\pi$-subgroups through elementary reasonaments.

\begin{proposition}
\label{pi'}
Let $G = AB$ be a $\pi$-separable group. Then $\abs{x^G}$ is a $\pi'$-number for each $\pi$-element $x\in A\cup B$ of prime power order if and only if the Hall $\pi$-subgroups of $G$ are abelian. Moreover, if this occurs, then the $\pi$-length of $G$ is at most $1$.
\end{proposition}

\begin{proof}
We can work with $H=(H\cap A)(H\cap B)\in\hall{\pi}{G}$ such that $H\cap A\in\hall{\pi}{A}$ and $H\cap B\in\hall{\pi}{B}$ in virtue of Lemma \ref{prefact}. The converse of the first claim is clear by Remark \ref{values}. So let us prove that $H = (H\cap A)(H\cap B)$ is abelian when $\abs{x^G}$ is a $\pi'$-number for each $\pi$-element $x\in A\cup B$ of prime power order. Suppose that the assertion is false and let us take $G$ a minimal counterexample. Then $\rad{\pi'}{G}=1$ by minimality, and so $\ce{G}{\rad{\pi}{G}}\leqslant\rad{\pi}{G}$. Take a Sylow $q$-subgroup $Q$ of $H\cap A$. Then each $y\in Q$ satisfies by assumption that $y \in  \ce{G}{\rad{\pi}{G}}\leqslant\rad{\pi}{G}$. Since $H\cap A$ is generated by its Sylow subgroups, it follows that $H\cap A \leqslant \ce{G}{\rad{\pi}{G}}\leqslant\rad{\pi}{G}$ and analogously for $H\cap B$. Hence $H\leqslant \ce{G}{\rad{\pi}{G}}\leqslant \rad{\pi}{G}$, so $H$ is abelian. The last claim follows directly.
\end{proof}

\begin{example}
Without the $\pi$-separability hypothesis, the previous result is not true, even for a not necessarily factorised group: Let $G=J_4$ be a Janko group, and let $\pi=\{3\}$. Then all the $3$-elements of $G$ have conjugacy class size not divisible by $3$, although a Sylow $3$-subgroup is non-abelian. This example appears in \cite{NT}.
\end{example}

Now we prove a result related to the above theorems.

\begin{proposition}
\label{p'}
Let $G = AB$ be a $\pi$-separable group. Assume that a given prime $p$ does not divide $\abs{x^G}$ for each $\pi$-element $x\in A\cup B$ of prime power order. Then there exists a Sylow $p$-subgroup of $G$ which normalises some Hall $\pi$-subgroup of $G$.
\end{proposition}

\begin{proof}
We may assume clearly that $p\in \pi'$. Besides, by conjugacy and Lemma \ref{prefact}, we may work with $H = (H \cap A)(H \cap B)\in\hall{\pi}{G}$. Let $G$ be a counterexample of least possible order. If $\rad{\pi}{G}\neq 1$, then by minimality we get the thesis. Hence we necessarily have that $N := \rad{\pi'}{G}\neq 1$.

We claim that $p$ does not divide $\abs{N:\no{N}{H}}$. Certainly, we may suppose that $p\in \pi(N)$. Let $P_0\in\syl{p}{N}$, and $G = \no{G}{P_0}N$ in virtue of Frattini's argument. For $q\in\pi$, let $Q=(Q\cap H\cap A)(Q\cap H\cap B)=(Q\cap A)(Q\cap B)$ be a prefactorised Sylow $q$-subgroup of $H$. If $a\in Q\cap A$, then by hypotheses we get $a\in\ce{Q\cap A}{P_0^n}$ for some $n\in N$. It follows $$(Q\cap A)N \subseteq \ce{(Q\cap A)N}{P_0}N \subseteq (Q\cap A)N,$$ so $(Q \cap A)N = \ce{(Q\cap A)N}{P_0}N\leqslant \ce{QN}{P_0}N$. We can argue analogously with $Q\cap B$ and thus $QN = (Q\cap A)(Q\cap B)N =\ce{QN}{P_0}N \leqslant \ce{HN}{P_0}N \leqslant HN$. Now for any $h\in H$, we also have $Q^hN=(QN)^h\leqslant (\ce{HN}{P_0}N)^h=\ce{HN}{P_0^h}N$. But $h\in G=\no{G}{P_0}N$, so we may assume $h\in N$ and so $Q^hN\leqslant\ce{HN}{P_0^h}N=(\ce{HN}{P_0}N)^h=\ce{HN}{P_0}N$. Since this is valid for each $q\in \pi$ we deduce $HN=\ce{HN}{P_0}N$. But $N$ is a $\pi'$-group, so there exists $n\in N$ such that $H \leqslant \ce{HN}{P_0^n}$. Hence $P_0^n \leqslant \ce{N}{H} \leqslant \no{N}{H} \leqslant N$. As $P_0\in\syl{p}{N}$, it follows that $p$ does not divide $\abs{N:\no{N}{H}}$.

On the other hand, by minimality there exists a Sylow $p$-subgroup $P$ of G that $P\leqslant \no{G}{HN}$. Again Frattini's argument for Hall $\pi$-subgroups produces $NHP = NH\no{NHP}{H}= N\no{NHP}{H}$. Therefore $p$ does not divides $\abs{NHP:\no{NHP}{H}}=\abs{N:\no{N}{H}}$. Thus there is a Sylow $p$-subgroup of $HNP$ (which is a Sylow $p$-subgroup of $G$) that normalises $H$.
\end{proof}

\medskip

In particular, when $G = A = B$ and $\pi = \{q\}$ we partially get \cite[Theorem 4.1]{BFM}. It is worth to remark again that both Propositions \ref{pi'} and \ref{p'} hold for any arbitrary factorisation of a $\pi$-separable group $G=AB$.


\section{On conjugacy class sizes of \texorpdfstring{$\pi$}{pi}-elements}
\label{pi-elements}

The assumptions in Corollary \ref{pi-decompo_ZGS} imply that the elements in the centre of a Hall $\pi$-subgroup $H$ of a group $G$ have to be central in $G$. Thus, a more general approach is to consider only the elements in $H\smallsetminus \ze{H}$, as Berkovich and Kazarin did through elementary arguments in \cite[Supplement to Theorem 1]{BK} for $\pi$-separable groups. For the sake of completeness, we present a proof of that result for groups which have a Hall $\pi$-subgroup (see Lemma \ref{BKlemma} below).

\begin{lemma}
\label{generated}
Let $H$ be a proper subgroup of a group $G$. Then $G = \langle G\smallsetminus H\rangle$.
\end{lemma}

\begin{proof}
Assume $G\smallsetminus H = \{x_1, \ldots ,x_k\}$ for $k \geq 1$. Clearly $G$ is a disjoint union of this last set and $H$. Let $X := \langle x_1, \ldots , x_k\rangle$, and let $z \in H$. Since for any $x_j \in G\smallsetminus H$ it holds that $zx_j \in G\smallsetminus H$, then $zx_j =x_s$ for some $s$, and it follows $z\in X$. Hence $H$ and $G \smallsetminus H$ are contained in $X$, which finishes the proof.
\end{proof}

\begin{lemma}
\label{BKlemma}
Let $G$ be a group with a non-abelian Hall $\pi$-subgroup $H$. Then $G$ is $\pi$-decomposable whenever every element in $H\smallsetminus \ze{H}$ has class size a $\pi$-number.
\end{lemma}

\begin{proof}
In virtue of Lemma \ref{wielandt} it follows that every element $x\in H\smallsetminus \ze{H}$ lies in $\rad{\pi}{G}$, and Lemma \ref{generated} leads to $H=\langle H\smallsetminus \ze{H}\rangle \leqslant\rad{\pi}{G}$. So $G$ has a normal Hall $\pi$-subgroup and it is $\pi$-separable. Let $F$ be a Hall $\pi'$-subgroup of $G$. If $g\in H\smallsetminus \ze{H}$, then by hypotheses $g\in \ce{H}{F^x}$ for some $x\in H$ since $G=HF$. Thus $H\subseteq \cup_{x\in H} (\ze{H}\ce{H}{F})^x \subseteq H$, so $H=\ze{H}\ce{H}{F}$ and $\ce{H}{F}$ is normal in $H$. Thus, every element $g\in H\smallsetminus \ze{H}$ lies in $\ce{H}{F}$. Since $H=\langle H\smallsetminus \ze{H}\rangle\leqslant \ce{H}{F}$, it follows $G=HF=H\times F$, as desired.
\end{proof}

\begin{example}
In view of the previous section, one might wonder whether the hypotheses in Lemma \ref{BKlemma} can be restricted to only prime power order $\pi$-elements. However, this is simply not possible:

Let $G$ be the direct product of a symmetric group of degree 3 and a non-abelian group of order 55, and let $\pi=\{2,3,11\}$. Then $H\in\hall{\pi}{G}$ is clearly non-abelian, $G$ is not $\pi$-decomposable, and $\abs{x^G}$ is a $\pi$-number for every element $x\in H\smallsetminus \ze{H}$ of prime power order.
\end{example}

Our next objective is to generalise Lemma \ref{BKlemma} for $\pi$-separable groups with a core-factorisation.

\begin{theorem}
\label{pi-non-central}
Let $G = AB$ be a core-factorisation, and suppose that $G$ is $\pi$-separable. Let $H = (H\cap A)(H\cap B)$ be a Hall $\pi$-subgroup of $G$ such that $H\cap X\in\hall{\pi}{X}$ for all $X\in \{A, B\}$, and assume that $H$ is non-abelian. Then the following statements are equivalent:
\begin{enumerate}
\item[\emph{(1)}] Every element in $((H\cap A)\cup (H\cap B))\smallsetminus \ze{H}$ has $G$-class size a $\pi$-number. 

\item[\emph{(2)}] For each $X\in \{A, B\}$, it follows that either $H\cap X\leqslant \ze{H}$ or $H\cap X \leqslant \ce{H}{F}$ for every $F\in\hall{\pi'}{G}$.
\end{enumerate}
Furthermore, if $H\cap X\leqslant\ze{H}$, then $X$ has $\pi$-length at most 1; and if $H\cap X \leqslant \ce{H}{F}$ for every $F\in\hall{\pi'}{G}$, then $X$ is $\pi$-decomposable.
\end{theorem}

\begin{proof}
Let $F=(F\cap A)(F\cap B)$ be a prefactorised Hall $\pi'$-subgroup of $G$ as in Lemma \ref{prefact}. Since $H\cap X\in\hall{\pi}{X}$, the last claim of the result follows from the fact that either $H\cap X$ is abelian, or $H\cap X\leqslant\ce{H\cap X}{F}\leqslant\ce{X}{F\cap X}$ being $F\cap X\in\hall{\pi'}{X}$. Moreover, the implication (2) $\Rightarrow$ (1) is clear. Therefore, it is enough to show that (1) $\Rightarrow$ (2). Notice that $H$ is non-abelian by assumption, so there exists some $X\in\{A, B\}$ such that $H\cap X\nleqslant \ze{H}$. Now, let us fix some arbitrary $F\in\hall{\pi'}{G}$, and note that $G=HF$. We split the proof in a number of steps.

\medskip

\underline{\sc{Step} 1:} If $H\cap X\nleqslant \ze{H}$ and $H\cap X$ is normal in $H$, then $H\cap X\leqslant \ce{H}{F}$.

\medskip

Let $\{X,Y\}=\{A,B\}$. We claim that $H=(H\cap X)\ce{H}{F}\ze{H}$, and we distinguish two cases. If $H\cap Y \leqslant \ze{H}$, then clearly $H=(H\cap X)\ze{H}= (H\cap X)\ce{H}{F}\ze{H}$. If $H\cap Y \nleqslant \ze{H}$, then we can pick $y \in (H\cap Y)\smallsetminus \ze{H}$. By our hypotheses, it follows that $y\in\ce{G}{F}^h$ for some $h\in H$, and hence $H\cap Y\subseteq \cup_{h\in H} \ce{H}{F}^h\ze{H}$. Since $H\cap X\unlhd H$, then $$ H \subseteq (H\cap X)\bigcup\limits_{h\in H} \ce{H}{F}^h\ze{H} \subseteq  \bigcup\limits_{h\in H} [(H\cap X)\ce{H}{F}\ze{H}]^h\subseteq H.$$ This fact yields $H= (H\cap X)\ce{H}{F}\ze{H}$. 

Now we choose $x\in (H\cap X)\smallsetminus \ze{H}$. Thus, we get $x\in\ce{H\cap X}{F^h}$ with $h\in (H\cap X)\ze{H}$. Indeed, $h=gz$ with $g\in H\cap X$ and $z\in \ze{H}$, so $x^{g^{-1}}=x^{h^{-1}}\in\ce{H\cap X}{F}$. We deduce $$H\cap X = \bigcup\limits_{g\in H\cap X} \ce{H\cap X}{F}^g (\ze{H}\cap X) = \bigcup\limits_{g\in H\cap X} [\ce{H\cap X}{F} (\ze{H}\cap X)]^g,$$ so $H\cap X=\ce{H\cap X}{F} (\ze{H}\cap X)$ and $\ce{H\cap X}{F}$ is normal in $H\cap X$. Now each element in $(H\cap X)\smallsetminus\ze{H}$ lies in $\ce{H\cap X}{F}$. Since $(H\cap X)\smallsetminus\ze{H} = (H\cap X)\smallsetminus (\ze{H}\cap X)$, in virtue of Lemma \ref{generated} we obtain $H\cap X=\langle (H\cap X)\smallsetminus\ze{H} \rangle\leqslant\ce{H}{F}$, as wanted.

\medskip

\noindent Now we assume without loss of generality that $H\cap A\nleqslant \ze{H}$. So the remainder of the proof aims to show that $H\cap A\leqslant\ce{H}{F}$.

\medskip

\underline{\sc{Step} 2:} We may suppose that neither $H\cap A$ nor $H\cap B$ are normal in $H$.

\medskip

By Step 1, we may assume that $H\cap A$ is not normal in $H$. If $H\cap B\leqslant\ze{H}$, then $H=(H\cap A)\ze{H}$ and $H\cap A\unlhd H$, a contradiction. Therefore $H\cap B\nleqslant \ze{H}$. If $H\cap B$ is normal in $H$, then by Step 1 it centralises $F$. So $H=(H\cap A)\ce{H}{F}$ and arguing similarly as in the last paragraph of Step 1, we can deduce that $H\cap A\leqslant\ce{H}{F}$.

\medskip

\underline{\sc{Step} 3:} If $N$ is a minimal normal subgroup of $G$, then $N$ is a $\pi$-group.

\medskip

Otherwise, we may assume that $N$ is a $\pi'$-group because $G$ is $\pi$-separable. We argue by induction on the order of $G$. We claim that the quotient $\overline{G}:=G/N$ inherits the hypotheses. Clearly we can assume $1\neq \overline{G}$, since $N = G$ implies the result trivially. Note that for $X\in\{A, B\}$, it holds $\overline{H\cap X} \leqslant \overline{H}\cap \overline{X}$ and, as $H\cap X\in\hall{\pi}{X}$, then $\overline{H\cap X} = \overline{H}\cap \overline{X}$. Thus $\overline{H}$ is prefactorised as in Lemma \ref{prefact}. Also $\overline{H}$ is non-abelian, $\overline{G}$ is a core-factorisation, and the class size condition is clearly inherited by quotients of $G$, so $\overline{G}$ satisfies the assumptions. 

By induction either $\overline{H}\cap \overline{X}\leqslant\ze{\overline{H}}$ or $\overline{H}\cap \overline{X}\leqslant\ce{\overline{G}}{\overline{F}}$ for all $X\in\{A, B\}$. If $\overline{H}\cap \overline{X}\leqslant \ze{\overline{H}}$, then $[H\cap X, H]\leqslant N\cap H=1$ and $H\cap X\leqslant\ze{H}$, a contradiction with Step 2. Therefore we necessarily have $[H\cap X, F]\leqslant N\leqslant F$, so $H\cap X$ normalises $F$. Since this is valid for all $X\in\{A, B\}$, we get that $F$ is normal in $G$. If $x\in (H\cap A)\smallsetminus \ze{H}$, then the fact that $\abs{x^G}$ is a $\pi$-number implies that $x\in\ce{H}{F}$. As $H\cap A$ is generated by the elements in $(H\cap A)\smallsetminus\ze{H}$, then $H\cap A$ centralises $F$, as wanted.

\medskip

\underline{\sc{Step} 4:} Conclusion.

\medskip

Since $G=AB$ is a core-factorisation, we can choose a minimal normal subgroup $N$ of $G$ which is covered by some $X\in\{A, B\}$. Moreover, $N$ is a $\pi$-group by the previous step. We consider $\overline{G}:=G/N$. If $\overline{H}$ is abelian, then $1\neq H' \leqslant N\leqslant H\cap X\leqslant H$, so $H\cap X$ is normal in $H$, which cannot happen because of Step 2. Thus, $\overline{G}$ inherits the hypotheses, and so $\overline{G}$ satisfies the thesis by induction on $\abs{G}$.

Now if $\overline{H\cap X}\leqslant\ze{\overline{H}}$, then $[H\cap X, H\cap Y]\leqslant N\leqslant H\cap X$, so $H\cap X$ is normal in $H$, a contradiction again with Step 2. Therefore, both $\overline{H\cap A}$ and $\overline{H\cap B}$ centralise $\overline{F}$, and it follows that $\overline{H}$ centralises $\overline{F}$. Hence $FN$ normal in $G$, and for all $g\in G$ there is some $n\in N$ such that $F^g=F^n$. 

Next we claim that $N = (\ze{H}\cap N)\ce{N}{F}$. If $N\leqslant \ze{H}$ then the claim is clear. If $N\nleqslant \ze{H}$, then we can take $m\in N\smallsetminus\ze{H}$ and by assumptions $m\in \ce{N}{F}^n$ for some $n\in N$. Hence $N = \cup_{n\in N} [(\ze{H}\cap N)\ce{N}{F}]^n$ and so $N = (\ze{H}\cap N)\ce{N}{F}$. 

Consequently, since each element $x\in (H\cap A) \smallsetminus \ze{H}$ lies in $\ce{H}{F}^n$ for some $n\in N$ and $N = (\ze{H}\cap N)\ce{N}{F}$, it follows $x\in \ce{H}{F}$. Thus $H\cap A = \langle (H\cap A) \smallsetminus \ze{H}\rangle \leqslant \ce{H}{F}$. 

Finally, we can argue analogously with $H\cap B$ in case that $H\cap B\nleqslant \ze{H}$. The result is now proved.
\end{proof}

\begin{example}
\label{BK-no-piseparable}
In contrast to Lemma \ref{BKlemma}, which relaxes the $\pi$-separability assumption in \cite[Supplement to Theorem 1]{BK}, we show that this condition is necessary in Theorem \ref{pi-non-central}: Let $G=A\times B$ be the direct product of $A=J_4$ a Janko group and $B=C_3$ a cyclic group of order 3, and let $\pi=\{3\}$. Note that this is clearly a core-factorisation, and $G$ is not $3$-separable. Moreover, if we take $P\in\syl{3}{G}$ such that $P=(P\cap A)(P\cap B)$ with $P\cap A\in\syl{3}{A}$ and $B=P\cap B\in\syl{3}{B}$, then $P$ is non-abelian and all the elements $x\in ((P\cap A)\cup (P\cap B))\smallsetminus\ze{P}=(P\cap A)\smallsetminus\ze{P}$ have $\abs{x^G}$ not divisible by $3$. However, neither $P\cap A$ is central in $P$ nor $P\cap A$ centralises every Hall $3'$-subgroup of $G$.
\end{example}

\begin{example}
The following example shows that in Theorem \ref{pi-non-central} we cannot affirm that $G$ is $\pi$-decomposable, in contrast to Lemma \ref{BKlemma}: Let $A$ be a dihedral group of order $8$ and let $B$ be a dihedral group of order $10$, and consider $\pi = \{2\}$. Then $G = A \times B$ satisfies the hypotheses in Theorem \ref{pi-non-central} but clearly it is not $2$-decomposable.
\end{example}

As a consequence, we obtain the next result.

\begin{theoremletters}
\label{pi-pi'}
Let $G = AB$ be a core-factorisation, and suppose that $G$ is $\pi$-separable. Let $H = (H\cap A)(H\cap B)$ be a Hall $\pi$-subgroup of $G$ such that $H\cap X\in\hall{\pi}{X}$ for all $X\in \{A, B\}$. Then the next assertions are pairwise equivalent:
\begin{enumerate}
	\item[\emph{(1)}] Every element in $(H\cap A)\cup (H\cap B)$ has $G$-class size either a $\pi$-number or a $\pi'$-number.
	
	\item[\emph{(2)}] For each $X\in \{A, B\}$, it follows that either $H\cap X \leqslant \ce{H}{F}$ for every $F\in\hall{\pi'}{G}$ or $H\cap X\leqslant \ze{H}$. 
\end{enumerate}
In addition:
\begin{itemize}
	\item[\emph{(a)}] $H\cap X\leqslant \ze{H}$ if and only if all $\abs{x^G}$ are $\pi'$-numbers for $x \in H\cap X$. In this case the $\pi$-length of $X$ is at most 1.
	
	\item[\emph{(b)}] $H\cap X \leqslant \ce{H}{F}$ for every $F\in\hall{\pi'}{G}$ if and only if all $\abs{x^G}$ are $\pi$-numbers for $x \in H\cap X$. In this case $X$ is $\pi$-decomposable.
\end{itemize}
\end{theoremletters}

\begin{proof}
The implication (2) $\Rightarrow$ (1) is clear. Let us prove (1) $\Rightarrow$ (2). We may suppose that $H$ is non-abelian. We work by induction on the order of $G$, and we first claim $\rad{\pi'}{G}\neq 1$. Otherwise $\rad{\pi}{G}$ is self-centralising in $G$. If $x\in (H\cap A)\cup (H\cap B)$, then $\abs{x^G}$ is either a $\pi$-number or a $\pi'$-number. In the first case $x\in \rad{\pi}{G}$ because of Lemma \ref{wielandt}, and in the second case $x\in \ce{G}{\rad{\pi}{G}}\leqslant\rad{\pi}{G}$. Since this is valid for every element $x\in (H\cap A)\cup (H\cap B)$, it follows that $\rad{\pi}{G}=H$ is prefactorised. Thus, for each $X\in\{A, B\}$ the elements $x\in \rad{\pi}{G}\cap X$ with $\abs{x^G}$ a $\pi'$-number lie in $\ze{\rad{\pi}{G}}$. So any $\abs{x^G}$ is a $\pi$-number for the elements $x\in ((H\cap A)\cup (H\cap B))\smallsetminus\ze{H}$. Applying Theorem \ref{pi-non-central} we obtain for each $X\in\{A, B\}$ that either $H\cap X\leqslant\ze{H}$ or $H \cap X = \rad{\pi}{G} \cap X\leqslant \ce{G}{F}$ for every $\pi'$-Hall subgroup $F$ of $G$, as desired. It follows then $\rad{\pi'}{G}\neq 1$.

Now, by induction, we get that $\overline{G} := G/\rad{\pi'}{G}$ satisfies the thesis. Let $X\in \{A, B\}$, so we have either $\overline{H\cap X}=\overline{H}\cap \overline{X}\leqslant\ze{\overline{H}}$ or $\overline{H\cap X}=\overline{H}\cap \overline{X}\leqslant\ce{\overline{G}}{\overline{F}}$ for any $F\in\hall{\pi'}{G}$. The first case leads to $[H\cap X, H]\leqslant H\cap \rad{\pi'}{G}=1$, so $H\cap X \leqslant \ze{H}$ and we are done. Hence, let us suppose $\overline{H \cap X}\leqslant\ce{\overline{G}}{\overline{F}}$ and $H\cap X\nleqslant \ze{H}$. Now if $\abs{x^G}$ is a $\pi'$-number for some $x\in (H\cap X)\smallsetminus \ze{H}$, then $|\overline{x}^{\overline{G}}|$ so is too. But $\overline{x}\in \overline{H\cap X}\leqslant\ce{\overline{G}}{\overline{F}}$, and we get that $\overline{x}$ is central in $\overline{G}$. In particular, $[H, \langle x\rangle]\leqslant\rad{\pi'}{G}\cap H=1$, so $x\in\ze{H}$, a contradiction. Thus, any $\abs{x^G}$ is a $\pi$-number for the elements $x\in ((H\cap A)\cup (H\cap B))\smallsetminus\ze{H}$, so the thesis follows as an application again of Theorem \ref{pi-non-central}. This completes the proof of (1) $\Rightarrow$ (2).

Next we prove (a). For the first claim, certainly only the sufficient condition is in doubt. So let us suppose that all $\abs{x^G}$ are $\pi'$-numbers for $x \in H\cap X$. By (2), either $H\cap X\leqslant \ze{H}$ or $H\cap X \leqslant \ce{H}{F}$ for every $F\in\hall{\pi'}{G}$. In the first case we are done, and in the second case it follows that any element in $H\cap X$ is central in $G$, so $H\cap X\leqslant\ze{H}$ also. Moreover, the last assertion follows from the fact that $X$ has abelian Hall $\pi$-subgroups.

Finally we prove (b). Again, it is enough to show in the first claim the sufficient condition. Let us suppose that all $\abs{x^G}$ are $\pi$-numbers for $x \in H\cap X$. If the case $H\cap X\leqslant \ze{H}$ in (2) holds, then $H\cap X$ is central in $G$ and we are done. So $H\cap X \leqslant \ce{H}{F}$ for every $F\in\hall{\pi'}{G}$. Further, the last assertion can be deduced from the fact that $H\cap X$ centralises a prefactorised Hall $\pi'$-subgroup $F$ of $G$ as in Lemma \ref{prefact}, so $H\cap X\leqslant\ce{X}{F\cap X}$ where $F\cap X\in\hall{\pi'}{X}$.
\end{proof}

\medskip

When $G=A=B$ in the previous result, the corollary below follows.

\begin{corollaryletters}
\label{corPiPi'}
Let $G$ be a $\pi$-separable group. Then the following statements are pairwise equivalent:
\begin{enumerate}
	\item[\emph{(1)}] Each $\pi$-element $x\in G$ has class size either a $\pi$-number or a $\pi'$-number.

	\item[\emph{(2)}] Either $G$ is $\pi$-decomposable or it has abelian Hall $\pi$-subgroups and its $\pi$-length is at most $1$.
	
	\item[\emph{(3)}] For every $\pi$-element $x\in G$, either all $\abs{x^G}$ are $\pi$-numbers or they are all $\pi'$-numbers.
\end{enumerate}
\end{corollaryletters}

In \cite[Theorem 4]{D} (see the next theorem, which is a little reformulation) Dolfi characterised the so-called \emph{class-$\pi$-separable} groups, i.e. groups all of whose class sizes are either $\pi$-numbers or $\pi'$-numbers.

\begin{theorem}
\label{silvio}
A group $G$ is class-$\pi$-separable if and only if, up to abelian direct factors, one of the following two cases happens:
\begin{enumerate}
 \item[\emph{(1)}] $G$ is either a $\pi$-group or a $\pi'$-group.
 
 \item[\emph{(2)}] Up to interchanging $\pi$ and $\pi'$, $G=HL$ with $H\in\hall{\pi}{G}$, $L\in\hall{\pi'}{G}$, $L\unlhd G$, both $H$ and $L$ are abelian, and $G/\rad{\pi}{G}$ is a Frobenius group. Indeed, $\rad{\pi}{G}=\ze{G}$, the set of the class sizes of $G$ is $\{1, \abs{H/\rad{\pi}{G}}, \abs{L}\}$, and $G$ is soluble.
\end{enumerate}
\end{theorem}

Motivated by Dolfi's result, we introduce the following factorised-group version of the concept of class-$\pi$-separability.

\begin{defin}
Let $G=AB$ be the product of two subgroups $A$ and $B$. We say that $G=AB$ is a \textbf{class-$\pi$-separable factorisation} whenever $\abs{x^G}$ is either a $\pi$-number or a $\pi'$-number for every element $x\in A\cup B$.
\end{defin}

Certainly, $G=AB$ is a class-$\pi$-separable factorisation if and only if it is a class-$\pi'$-separable factorisation. Besides, any central product of two class-$\pi$-separable groups provides a class-$\pi$-separable factorisation.

We cannot assert in a class-$\pi$-separable factorisation $G=AB$, a priori, that both $A$ and $B$ are class-$\pi$-separable groups. This is because, for $x\in A$, there is no relation in general between the sets $\pi(\abs{x^A})$ and $\pi(\abs{x^G})$. Nevertheless, under the additional assumption of being a core-factorisation, we determine in Theorem \ref{teosilvio} that this phenomenon actually occurs. To prove that fact we need firstly some preparation. The next result generalises Lemma \ref{wielandt}.

\begin{lemma}
\label{lemma_beltranfelipe}
Let $G$ be a $\pi$-separable group. If $\abs{x^G}$ is a $\pi$-number for some $x\in G$, then $(\langle x^G\rangle)'$ is a $\pi$-group. In particular, $x\in\rad{\pi,\pi'}{G}$.
	
Indeed, if $\pi$ consists of a single prime $q$, then the same statement is valid even if $G$ is not $q$-separable.
\end{lemma}

\begin{proof}
The first claim is exactly \cite[Theorem C]{BF} (see also \cite[Lemma J(k)]{BK}), whilst the second assertion is \cite[Lemma 3]{BK}.
\end{proof}

\medskip

There are easy examples which illustrates that the first claim is not true when the $\pi$-separability hypothesis is removed (\cite{BF}).

The following well-known result is due to It\^{o}. 

\begin{lemma}\emph{\cite[Proposition 5.1]{I}}
\label{ito}
Let $G$ be a group. Suppose that $p$ and $q$ are distinct primes that divide two different conjugacy class sizes of $G$, but there is no $g\in G$ with $pq$ dividing $\abs{g^G}$. Then $G$ is either $p$-nilpotent or $q$-nilpotent.
\end{lemma}

In relation to Theorem \ref{pi-pi'}, when we consider all the elements in the factors (not just those of order a $\pi$-number), we obtain the proposition below. Actually, this generalises \cite[Lemma 6]{D}.

\begin{proposition}
\label{separa-criterion}
Let $G=AB$ be the product of the subgroups $A$ and $B$, and assume that $G=AB$ is both a core-factorisation and a class-$\pi$-separable factorisation. Then $G$ is $\pi$-separable.
\end{proposition}

\begin{proof}
Since $G=AB$ is a core-factorisation, there exists a chief series $1=N_0 \unlhd N_1 \unlhd \cdots \unlhd N_{n-1} \unlhd N_n=G$  with each chief factor covered by either $A$ or $B$. In fact, we can refine that series in order to get a composition series whose factors are covered by either $A$ or $B$. Thus, for each $1\leq i \leq n$, there exist subgroups $T_j$ such that $N_{i-1}=T_0\unlhd T_1\unlhd T_2\unlhd \cdots \unlhd T_m= N_i$ and $T_j/T_{j-1}$ is simple for every $1\leq j \leq m$. We claim that each of these  $T_j/T_{j-1}$ is either a $\pi$-group or a $\pi'$-group, and so $G$ will be $\pi$-separable. Note that $T_j/T_{j-1}$ is isomorphic to $(T_j/N_{i-1})/(T_{j-1}/N_{i-1})$. Moreover $T_j/N_{i-1}$ is subnormal in $N_i/N_{i-1}$, which is normal in $G/N_{i-1}$ and it is covered by either $A$ or $B$, as $G/N_{i-1}=(AN_{i-1}/N_{i-1})(BN_{i-1}/N_{i-1})$ is a core-factorisation. Then all the class sizes of $N_i/N_{i-1}$, and so all the class sizes of $T_j/T_{j-1}$, are either $\pi$-numbers or $\pi'$-numbers. If there are two primes $p\in \pi$ and $q \in \pi'$ that divide two different class sizes of $T_j/T_{j-1}$, as $pq$ does not divide any class size of $T_j/T_{j-1}$, applying Lemma \ref{ito} we get that the simple group $T_j/T_{j-1}$ has either a normal $p$-complement or a normal $q$-complement. We deduce that either $p$ or $q$ does not divide the order of $T_j/T_{j-1}$, a contradiction. Thus, we may assume that each prime $q\in\pi'$ does not divide any class size of $T_j/T_{j-1}$, so it has a central Sylow $q$-subgroup. It follows that $q\notin\pi(T_j/T_{j-1})$ for every $q\in \pi'$, so $T_j/T_{j-1}$ is a $\pi$-group and $G$ is $\pi$-separable.
\end{proof}

\medskip

We are now ready to prove Theorem \ref{teosilvio}. We will use mainly Theorem \ref{pi-pi'} and some of Dolfi's techniques in \cite{D}.

\begin{theoremletters}
\label{teosilvio}
Let $G=AB$ be the product of the subgroups $A$ and $B$, and assume that $G=AB$ is both a core-factorisation and a class-$\pi$-separable factorisation. Then, up to abelian direct factors, one of the following two possibilities holds for any $X\in \{A, B\}$:

\begin{enumerate}
	\item[\emph{(1)}] $X$ is either a $\pi$-group or a $\pi'$-group.
	
	\item[\emph{(2)}] Up to interchanging $\pi$ and $\pi'$, $X=X_{\pi}X_{\pi'}$ where $X_{\pi}\in\hall{\pi}{X}$ and $X_{\pi'}\in\hall{\pi'}{X}$, $X_{\pi'}\unlhd X$, both $X_{\pi}$ and $X_{\pi'}$ are abelian, and $X/\rad{\pi}{X}$ is a Frobenius group. Indeed, $\rad{\pi}{X}=\ze{X}$, the class sizes of $X$ are $\{1, \abs{X_{\pi}/\rad{\pi}{X}}, \abs{X_{\pi'}}\}$, and $X$ is soluble. 
\end{enumerate}
In particular, both $A$ and $B$ are class-$\pi$-separable groups.
\end{theoremletters}

\begin{proof}
Observe that $G$ is $\pi$-separable by Proposition \ref{separa-criterion}. Take $H=(H\cap A)(H\cap B)\in\hall{\pi}{G}$ and $F=(F\cap A)(F\cap B)\in\hall{\pi'}{G}$ with $H\cap X\in\hall{\pi}{X}$ and $F\cap X\in\hall{\pi'}{X}$ for all $X\in \{A, B\}$ as in Lemma \ref{prefact}. Set $X_{\pi}:=H\cap X$ and $X_{\pi'}:=F\cap X$. Certainly, $X=X_{\pi}X_{\pi'}$. Let us analyse the structure of any $X\in\{A, B\}$. We may assume that $X$ has no abelian direct factors. We proceed in five steps.

\medskip

\underline{\sc{Step} 1:} Let $\sigma\in\{\pi, \pi'\}$. If every $G$-class size of elements in $X$ is a $\sigma$-number, then $X$ is a $\sigma$-group.

\medskip

Applying Theorem \ref{pi-pi'} for the elements in $X_{\sigma}$ and in $X_{\sigma'}$, we deduce $X=X_{\sigma} \times X_{\sigma'}$ with $X_{\sigma'}$ abelian. Then $X_{\sigma'}$ is an abelian direct factor of $X$, so $X_{\sigma'}=1$ and $X$ is a $\sigma$-group, as wanted.

\medskip

\noindent In particular, we may assume in the sequel that $X$ is nor a $\pi$-group nor a $\pi'$-group, and that there exist $x, y\in X$ such that $\pi(\abs{x^G})$ contains a prime in $\pi$ and $\pi(\abs{y^G})$ contains a prime in $\pi'$, respectively.

\medskip

\underline{\sc{Step} 2:} $X$ has both abelian Hall $\pi$-subgroups and Hall $\pi'$-subgroups.

\medskip

In virtue of Theorem \ref{pi-pi'}, we get that either every element in $X_{\pi}$ has $G$-class size a $\pi$-number or every element in $X_{\pi}$ has $G$-class size a $\pi'$-number. In the first case we get $X=X_{\pi} \times X_{\pi'}$ and, as we are assuming that $X_{\pi'}\neq 1$, it cannot be an abelian direct factor, so necessarily there is a non-trivial element $y\in X_{\pi'}$ with $\abs{y^G}$ a $\pi'$-number. Hence, for any $x\in X_{\pi}\smallsetminus \ze{X_{\pi}}$ we get that $\abs{(xy)^G}$ is neither a $\pi$-number nor a $\pi'$-number, a contradiction. Hence all $\abs{x^G}$ are $\pi'$-numbers for the elements $x\in X_{\pi}$  and analogously all $\abs{y^G}$ are $\pi$-numbers for the elements $y\in X_{\pi'}$. Now Theorem \ref{pi-pi'} (a) yields that $X$ has abelian Hall $\pi$-subgroups and Hall $\pi'$-subgroups, as wanted.

\medskip

\noindent Note that our class size assumptions imply $X=\ce{X}{\rad{\pi}{G}}\cup\ce{X}{\rad{\pi'}{G}}$, so we may assume $[X, \rad{\pi}{G}]=1$ in the remainder of the proof.

\medskip

\underline{\sc{Step} 3:} $X_{\pi'}$ is normal in $X$. In particular, $X$ is soluble.

\medskip

Denoting $\overline{G}:=G/\rad{\pi}{G}$, in virtue of Lemma \ref{lemma_beltranfelipe} it follows $\overline{X_{\pi'}}\leqslant \rad{\pi'}{\overline{G}}\cap \overline{X}\leqslant \rad{\pi'}{\overline{X}}.$ Since $X_{\pi'}\in\hall{\pi'}{X}$, we get $\overline{X_{\pi'}}=\rad{\pi'}{\overline{X}}$, and then $X_{\pi'}\rad{\pi}{G}$ is normalised by $X$. Now for any $x\in X$, we deduce that $X_{\pi'}^x\leqslant (X_{\pi'}\rad{\pi}{G})^x=X_{\pi'}\rad{\pi}{G}$, so there exists some $n\in \rad{\pi}{G}$ such that $X_{\pi'}^x=X_{\pi'}^n$. As $[X, \rad{\pi}{G}]=1$, then $X_{\pi'}\unlhd X$, and $X$ is soluble.

\medskip

\underline{\sc{Step} 4:} $\rad{\pi}{X}=\ze{X}$.

\medskip

Since $X_{\pi'}$ is abelian and normal in $X$, by coprime action, we deduce $X_{\pi'}=[X_{\pi'}, X_{\pi}]\times\ce{X_{\pi'}}{X_{\pi}}$. Note that $\ce{X_{\pi'}}{X_{\pi}}\leqslant \ze{X}$, so $\ce{X_{\pi'}}{X_{\pi}}$ is an abelian direct factor of $X$ and we may assume $\ce{X_{\pi'}}{X_{\pi}}=1$. Hence $X_{\pi'}\cap \ze{X}=1$ and $\ze{X}\leqslant \rad{\pi}{X}$. The other inclusion is clear because $X_{\pi'}$ is normal in $X$ and $X_{\pi}$ is abelian.

\medskip

\underline{\sc{Step} 5:} $X/\rad{\pi}{X}$ is a Frobenius group. In particular, the set of class sizes of $X$ is $\{1, \abs{X_{\pi}/\rad{\pi}{X}}, \abs{X_{\pi'}}\}$.

\medskip

Set $\tilde{X}:=X/\rad{\pi}{X}$. We claim that $\tilde{X_{\pi}}$ acts fixed-point-freely on $\tilde{X_{\pi'}}$. Take $1\neq\tilde{y}\in \tilde{X_{\pi'}}$ and $\tilde{x}\in \tilde{X_{\pi}}$ such that $[\tilde{y}, \tilde{x}]=1$. Then $[y, x]\in \rad{\pi}{X}$ and it is a $\pi'$-element since $X_{\pi'}\unlhd X$. Now $xy=yx$ and both $\abs{x^G}$ and $1\neq \abs{y^G}$ divides $\abs{(xy)^G}$. It follows necessarily that $x\in\ze{G}\cap X\leqslant\ze{X}=\rad{\pi}{X}$ so $\tilde{x}=1$ and we are done. 

Finally, note that $\rad{\pi}{X}=\ze{X}$ implies $\ze{X}\cap X' \leqslant \ze{X}\cap X_{\pi'}=1$. This fact leads to $\ce{X}{g}/\ze{X}=\ce{X/\ze{X}}{g\ze{X}}$ for all $g\in X$, so the class sizes of $X$ and $X/\ze{X}$ coincides. Since $X/\ze{X}$ is a Frobenius group, then the set of class sizes of $X$ is $\{1, \abs{X_{\pi}/\rad{\pi}{X}}, \abs{X_{\pi'}}\}$.

To conclude, from the described structure of $X$, we get that $X$ is a class-$\pi$-separable group.
\end{proof}

\medskip

Observe that \cite[Theorem 4]{D} (Theorem \ref{silvio} above) is now a direct consequence of the previous theorem when $G=A=B$.

\begin{example}
A core-factorisation of two class-$\pi$-separable groups might not be a class-$\pi$-separable factorisation: Let $M$ be the direct product of a non-abelian group of order 21 and a cyclic group of order 5. Consider the action of a cyclic group $N$ of order 2 on $M$, in such way that $N$ acts trivially on the elements of order 3. If we take $\pi:=\{3, 7\}$, $A\in\hall{\pi}{G}$ and $B\in\hall{\pi'}{G}$, then $G=AB$ is a core-factorisation since $G$ is $\pi$-separable (indeed it is soluble). Certainly, $G=AB$ is not a class-$\pi$-separable factorisation, although $A$ and $B$ are class-$\pi$-separable groups.
\end{example}

The next example illustrates that we cannot expect to obtain in a class-$\pi$-separable factorisation, for instance, either that the Hall $\pi$- and $\pi'$-subgroups of $G$ are abelian or that $G$ is soluble.

\begin{example}
Let $A$ be an alternating group of degree $5$ and $B$ the semidirect product of a cyclic group of order 29 acted on by a cyclic group of order 7. Consider $G=A\times B$, and $\pi=\{2,3,5\}$. Clearly, every element in $A\cup B$ has class size equal to a $\pi$-number or a $\pi'$-number, but $G$ is not soluble. Moreover, neither the Hall $\pi$-subgroup nor the Hall $\pi'$-subgroup of $G$ is abelian.
\end{example}

To conclude, inspired by \cite{BF}, we concentrate on factorised groups whose $\pi$-elements in the factors have conjugacy class lengths equal to prime powers. Indeed, in \cite{FMOprime} we analysed products of groups where, for a given prime $p$, the $p$-elements in the factors have prime power class sizes.

\begin{proposition}\emph{\cite[Theorems A and B]{FMOprime}}
\label{pBaer}
Let $G=AB$ be the product of the subgroups $A$ and $B$, and let $P\in\syl{p}{G}$. Assume that $\abs{g^G}$ is equal to a prime power for each $p$-element $g\in A\cup B$. Then we have:
	\begin{itemize}
		\item[\emph{(1)}]  $P\fit{G}$ is normal in $G$.
		
		\item[\emph{(2)}]  There exist unique primes $q$ and $r$ such that $\abs{x^G}$ is a $q$-number for every $p$-element $x\in A$, and $\abs{y^G}$ is an $r$-number for every $p$-element $y\in B$, respectively. (Eventually, $p\in\{q, r\}$ or $q = r$.)
	\end{itemize} 
\end{proposition}

Then, we will use the above proposition for the primes $p\in \pi$, joint with the next significant lemma and also some stated results in previous sections in order to prove Theorem \ref{teoprime} below.

\begin{lemma}
\label{lemmaBKprime}
Let $G$ be a group, and let $x, y\in G\smallsetminus \ze{G}$ be $\pi$-elements such that $\abs{x^G}$ and $\abs{y^G}$ are two distinct prime powers, and assume that $\abs{(xy)^G}$ is also a prime power. Then $\langle x, y \rangle^G\leqslant \rad{\pi}{G}$, and $\abs{(xy)^G}=\op{max}\{\abs{x^G}, \abs{y^G}\}$ is a power of a prime $q\in \pi$. In particular, if $\hall{\pi}{G}\neq \emptyset$, then a Hall $\pi$-subgroup of $G$ is non-abelian.
\end{lemma}

\begin{proof}
It is enough to mimic the proof of \cite[Lemma 4]{BK} with $\pi$ instead of $p$.
\end{proof}

\begin{theoremletters}
\label{teoprime}
Let $G=AB$ be a core-factorisation. Suppose that $\abs{x^G}$ is a prime power for every $\pi$-element $x\in A\cup B$. Then $G$ is $\pi$-separable of $\pi$-length at most $1$. Moreover, for each $X\in\{A, B\}$, one of the following two possibilities holds:
\begin{enumerate}
	\item[\emph{(1)}] All $\abs{x^G}$ are powers of a fixed prime $q$ for every $\pi$-element $x\in X$. In addition: 
	\begin{enumerate}
		\item[\emph{(a)}] If $q\notin \pi$, then $X$ has an abelian Hall $\pi$-subgroup $X_{\pi}$. In this case $X_{\pi}\rad{q}{G}$ is normalised by $X$.
		
		\item[\emph{(b)}] If $q\in \pi$, then $X$ is $\pi$-decomposable with nilpotent Hall $\pi$-subgroup $X_{\pi}$ and the Sylow subgroups of $X_{\pi}$ are all abelian except possibly for the prime $q$.
	\end{enumerate}
	
	\item[\emph{(2)}] All $\abs{x^G}$ are powers of two distinct fixed primes $q$ and $r$, both in $\pi$, for every $\pi$-element $x\in X$. In this case, $X$ is $\pi$-decomposable, and the Hall $\pi$-subgroup $X_{\pi}$ of $X$ satisfies that $X_{\pi}/\ze{X_{\pi}}$ is a Frobenius group with abelian kernel and complement of orders a $q$-power and an $r$-power, respectively.
\end{enumerate}
\end{theoremletters}

\begin{proof}
First of all, we prove the assertion on the $\pi$-separability of $G$. Applying Proposition \ref{pBaer} (1) for each prime $p\in \pi$, we can affirm that $G/\fit{G}$ has a normal Sylow $p$-subgroup. Therefore, $G$ is $p$-separable with $p$-length at most $1$ for each prime $p\in \pi$ and, in particular, it is $\pi$-separable with $\pi$-length at most $1$. Henceforth, we can take with $H=(H\cap A)(H\cap B)\in\hall{\pi}{G}$ with $H\cap A\in\hall{\pi}{A}$ and $H\cap B\in\hall{\pi}{B}$. 

Next we assume that every $\abs{x^G}$ is a power of a fixed prime $q$ for the $\pi$-elements in $H\cap X$, as in case (1). If $q\notin \pi$, then we obtain that $X$ has an abelian Hall $\pi$-subgroup $X_{\pi}:=H\cap X$ in virtue of Theorem \ref{pi-pi'} (a). Let us prove that $X_\pi\rad{q}{G}$ is normalised by $X$. Let us denote by bars the quotients over $\rad{q}{G}$. Thus, Lemma \ref{lemma_beltranfelipe} applied to each element $x\in X_{\pi}$ yields that $\overline{X_{\pi}}\leqslant \rad{q'}{\overline{G}}\cap \overline{X}\leqslant \rad{q'}{\overline{X}}.$ Since $q\notin \pi$ and $X_{\pi}\in\hall{\pi}{X}$, then $\overline{X_{\pi}}=\rad{q'}{\overline{X}}$. Now clearly $X_{\pi}\rad{q}{G}$ is normalised by $X$, as wanted in (a).

If $q\in \pi$, then Theorem \ref{pi-pi'} (b) provides that $X=X_{\pi} \times \rad{\pi'}{X}$, so it remains to show that $X_{\pi}$ is nilpotent with abelian Sylow subgroups, except possibly for the prime $q$. Recall that $G$ is $p$-separable for every prime $p\in \pi$. Hence, Theorem \ref{pi-pi'} (b) applied for the prime $q$ gives $X=X_q\times \rad{q'}{X}$, so $X_{\pi}=X_q\times \rad{\sigma}{X}$ where $\sigma:=\pi \smallsetminus \{q\}$. Finally, $\rad{\sigma}{X}$ is abelian in virtue again of Theorem \ref{pi-pi'} (a) applied for $\sigma$, and (b) is proved.

\medskip

\noindent From now on, we assume that (1) does not hold and we demonstrate case (2) in three steps.

\medskip

\underline{\sc{Step} 1:} At most two different primes appear as divisors of the class sizes of the $\pi$-elements in $X$.

\medskip

Assuming the contrary, there exists three non-central $\pi$-elements in $X$, say $x_1, x_2$ and $x_3$, such that their $G$-class sizes are equal to powers of three different primes, say $p_1, p_2$ and $p_3$, respectively. All the $x_i$ decompose as product of commuting prime power order ($\pi$-)elements, so Proposition \ref{pBaer} (2) joint with this last fact allow us to suppose that the orders of the $x_i$ are coprime prime powers. Hence, from now on we assume that $x_i$ is a $q_i$-element with $\abs{x_i^G}$ equal to a $p_i$-power, for each $i\in\{1,2,3\}$. Since either $p_1\neq q_2$ or $p_1\neq q_3$, we may assume the first case and so there exists $g\in G$ such that $Q^g\leqslant \ce{G}{x_1}$, where $Q=(Q\cap A)(Q\cap B)\in\syl{q_2}{G}$,  $Q\cap A\in\syl{q_2}{A}$ and $Q\cap B\in\syl{q_2}{B}$. We may suppose $x_2\in Q\cap X$ by Remark \ref{values}, and so we get $x_2^g\in \ce{G}{x_1}$. But $(\abs{x_1^G}, \abs{x_2^G})=1$, so $G=\ce{G}{x_1}\ce{G}{x_2}$ and we obtain $x_2\in \ce{G}{x_1}$. Now $x_1x_2=x_2x_1\in X$ is a $\pi$-element, and it follows that $\abs{(x_1x_2)^G}$ is divisible by both $p_1$ and $p_2$, a contradiction.

\medskip

\underline{\sc{Step} 2:} Assuming that all $\abs{x^G}$ are powers of two distinct fixed primes $q$ and $r$ for every $\pi$-element $x\in X$, we claim that $\{q, r\}\subseteq \pi$.

\medskip

Let $x$ and $y$ be $\pi$-elements in $X$ such that $\abs{x^G}$ is a non-trivial $q$-power and $\abs{y^G}$ is a non-trivial $r$-power. Again by Remark \ref{values}, we can assume without loss of generality that $x, y\in H\cap X$. Hence $xy\in H\cap X$ and $\abs{(xy)^G}$ is a prime power also. Thus, in virtue of Lemma \ref{lemmaBKprime} we have that the prime which corresponds to the largest class size lies in $\pi$. So let us suppose that the largest one is $\abs{x^G}$, that is, $q\in \pi$. If $r\notin\pi$, then there exists $g\in G$ such that $x^g\in H^g\leqslant \ce{G}{y}$. Also $G=\ce{G}{x}\ce{G}{y}$, so $xy=yx\in H\cap X$ and we conclude that $\abs{(xy)^G}$ is divisible by both $q$ and $r$, a contradiction.

\medskip

\underline{\sc{Step} 3:} $X$ is $\pi$-decomposable, and the Hall $\pi$-subgroup $X_{\pi}$ of $X$ satisfies that $X_{\pi}/\ze{X_{\pi}}$ is a Frobenius group with abelian kernel and complement of orders a $q$-power and an $r$-power, respectively.

\medskip

Since we are assuming that all $\abs{x^G}$ are powers of two distinct fixed primes $q$ and $r$, both in $\pi$, for every $\pi$-element $x\in X$, then by Theorem \ref{pi-pi'} we get that $X_{\pi}$ centralises every Hall $\pi'$-subgroup of $G$. Indeed it is $\pi$-decomposable. For proving the remaining assertion, we distinguish two cases on the class sizes of the $\pi$-elements in $Y$, where $\{X, Y\}=\{A, B\}$: either they are powers of a prime in $\pi$ or in $\pi'$. In the second case, by Theorem \ref{pi-pi'} we obtain $Y_{\pi}:=H\cap Y\leqslant \ze{H} \leqslant \ce{H}{X_{\pi}}$. As $X_{\pi}$ centralises every Hall $\pi'$-subgroup of $G$, it follows that $X_{\pi}$ is normal in $G$. Then all the $X_{\pi}$-class sizes of elements in $X_{\pi}$ are either $q$-powers or $r$-powers, and Theorem \ref{silvio} yields the desired structure of $X_{\pi}/\ze{X_{\pi}}$. In the other case, $Y_{\pi}$ also centralises every Hall $\pi'$-subgroup of $G$, so $H=X_{\pi}Y_{\pi}$ is normal in $G$. We deduce that the class sizes in $H$ of all elements in $X_{\pi}\cup Y_{\pi}$ are either $q$-powers or $r$-powers. But we may affirm that $H=X_{\pi}X_{\pi'}$ is a core-factorisation in virtue of Lemma \ref{prefact}, so Theorem \ref{teosilvio} applied to $H$ completes the proof of (2).
\end{proof}

\medskip

The main result of \cite{BF} now can be retrieved from the above theorem (see the corollary below). It is significant to notice that our proof, however, uses different tools.

\begin{corollary}
Let $G$ be a group for which every $\abs{x^G}$ is a prime power for the $\pi$-element $x \in G$. Then one of the following possibilities occurs.
\begin{itemize}
	\item[\emph{(a)}] All $\abs{x^G}$ are powers of a fixed prime $q$. Moreover,
	\begin{itemize}
		\item[\emph{(1)}] $q\notin \pi$ if and only if $G$ has an abelian Hall $\pi$-subgroup $H$. In this case, $H\rad{q}{G}\unlhd G$.
		
		\item[\emph{(2)}] $q\in \pi$ if and only if $G$ is $\pi$-decomposable with nilpotent Hall $\pi$-subgroup $H$, and the Sylow subgroups of $H$ are all abelian except possibly for the prime $q$.
	\end{itemize}
	\item[\emph{(b)}]  All $\abs{x^G}$ are powers of two distinct primes, say $q$ and $r$. This happens if and only if $\{q, r\}\subseteq \pi$, $G$ is $\pi$-decomposable, and the Hall $\pi$-subgroup $H$ of $G$ satisfies that $H/\ze{H}$ is a Frobenius group with abelian kernel and complement of orders a $q$-power and an $r$-power, respectively. 
\end{itemize}
Furthermore, in all cases, $G$ has $\pi$-length at most $1$.
\end{corollary}

The reader might think about the possibility of obtaining in Theorem \ref{teoprime} extra information about the $\pi$-structure of the whole group $G$. Nevertheless, in a factorised group, several different cases may occur with the primes appearing as class sizes of $\pi$-elements in the factors, so the $\pi$-structure is quite unrestricted:

\begin{example}
Let $A$ be a symmetric group of degree $3$ and $B$ the semidirect product of a cyclic group of order 5 acting on a cyclic group of order 6. Consider $G=A\times B$, and $\pi=\{2,3\}$. Clearly, the hypotheses in Theorem \ref{teoprime} are satisfied, but neither the Hall $\pi$-subgroup of $G$ is abelian (as in case (1)(a)) nor $G$ is $\pi$-decomposable (case (2)).
\end{example}



\end{document}